\documentclass[12pt]{article}

\usepackage{color}
\usepackage{graphics, graphpap, graphicx, verbatim}
\usepackage{natbib}
\usepackage{amsfonts}
\usepackage{amsmath,amsthm}

\newcommand{\be}[1] {\begin{equation} \label{#1}}
\newcommand{\ee} {\end{equation}}
\newcommand{\bea} {\begin{eqnarray}}
\newcommand{\eea} {\end{eqnarray}}
\newcommand{\beas} {\begin{eqnarray*}}
\newcommand{\eeas} {\end{eqnarray*}}

\newcommand{\st}{ \mathrel{\ooalign{$\ni$\cr\kern-1pt$-$\kern-6.5pt$-$}}}

\newtheorem{theorem}{Theorem}

\newcommand{\N}{\mathbb{N}}

\newcommand{\X}{\mathcal{X}}

\definecolor{DarkGreen}{rgb}{0,0.5,0}

\title{A Lower Bound on the Mixing Time of Uniformly Ergodic Markov Chains
in Terms of the Spectral Radius} \author{Dawn B.\ Woodard \\ School of Operations Research and Information
Engineering\\
and Department of Statistical Science\\ Cornell University
\footnote{Thanks to Aaron Smith at ICERM for his insightful suggestions.  This work was supported in part by National
Science Foundation Grant Number DMS-1209103. }}

\begin{document}

\maketitle

\abstract{

We give a bound on the mixing time of a uniformly ergodic, reversible
Markov chain in terms of the spectral radius of the transition kernel.
This bound has been established previously in finite state spaces, and
is widely believed to hold in general state spaces, but a proof has
not been provided to our knowledge.

\vspace{0.8cm}
  \noindent \textsc{\bf{Keywords: }} Markov chains, mixing time,
  spectral radius, spectral gap, general state space.  }

\vspace{1cm}

Consider a uniformly ergodic Markov chain with transition kernel $T$
on general (countably generated) state space $\X$, with stationary distribution $\pi$.
The distance of a Markov chain to stationarity is commonly measured using
the total variation norm, defined for a signed measure $\mu$ as
\begin{align*}
\|\mu\|_{TV} &= \sup_{A \subset \X} |\mu(A)|
\end{align*}
where the supremum is over measurable sets $A$.  The mixing time of the
Markov chain is the number of iterations required for the total
variation distance to stationarity to drop below a particular
threshold $\epsilon \in (0, \frac{1}{2})$,
for all initial states:
$$\tau_\epsilon \equiv  \min\{n: \sup_{x \in \X}\|T^n(x,\cdot) -
\pi\|_{TV} \leq \epsilon \}$$
\citep{aldo:1982}.  For $T$ uniformly ergodic we have $\tau_\epsilon <
\infty$ (Meyn and Tweedie 1993, Theorem 16.0.2\nocite{meyn:twee:1993}).

For $\X$ finite, $T$ reversible, and $\epsilon = {(2e)^{-1}}$,
Proposition 8(a) of \cite{aldo:1982} provides a lower bound on
$\tau_\epsilon$ in terms of the spectral radius of $T$.  The extension
to general $\epsilon$ is immediate, but the extension to general state
spaces is not.  We provide this extension in
Theorem~\ref{Thm:MixTime}.
\begin{theorem}\label{Thm:MixTime}
For a uniformly ergodic, reversible Markov chain $T$ on countably
generated state space $\X$, the mixing time $\tau_\epsilon$ satisfies
$$\tau_\epsilon \geq \frac{\ln (2\epsilon)}{\ln \rho}$$
where $\rho\in [0,1)$ is the $L_2(\pi)$ spectral radius of $T$ and
  $\epsilon \in (0,\frac{1}{2})$.  When $\rho = 0$ this is
  taken to mean that $\tau_\epsilon > 0$.
\end{theorem}
\noindent When $\rho$ is close to 1 we have $-\ln \rho \approx 1-\rho$, so that one
can obtain a corresponding lower bound on $\tau_\epsilon$ in terms of the
inverse of the spectral gap $(1-\rho)$.  In particular,
\begin{align}\label{Eqn:AltBound}
\tau_\epsilon \geq \frac{\rho\ln (2\epsilon)^{-1}}{2(1-\rho)}.
\end{align}

\begin{proof}[Proof of Theorem~\ref{Thm:MixTime}]
Define 
$$\Delta_n \equiv \sup_{x \in \X} \|T^n(x,\cdot) -
\pi\|_{TV}.$$
\cite{aldo:1982} shows at
the beginning of his Proposition 8(a) that when $\X$ is finite and
$\epsilon = (2e)^{-1}$,
\begin{align}\label{Eqn:DeltaLim2}
\lim_{n \rightarrow \infty}\Delta_n e^{n\left[-\ln (2\epsilon)\right]/\tau_\epsilon} \leq (2\epsilon)^{-1}.
\end{align}
However, his proof also holds for general state spaces and general
$\epsilon$.  It relies on the sub-multiplicative property of
$2\Delta_n$, which is shown in general state spaces, for instance, as
Proposition 3(e) of \cite{robe:rose:2004}.

We will show that for all $\delta >0$, 
\begin{align}\label{Eqn:DeltaLim}
\lim_{n \rightarrow \infty} \Delta_n e^{n (\frac{1}{\beta} +
  \delta)} > 0
\end{align}
where $\beta \equiv -\frac{1}{\ln \rho}$.  This implies $\lim_{n
  \rightarrow \infty} \Delta_n e^{n (\frac{1}{\beta} + \delta)} =
  \infty $ for all $\delta > 0$.  
  Combining with (\ref{Eqn:DeltaLim2}), we have that
  $\frac{1}{\beta} + \delta \geq \frac{-\ln (2\epsilon)}{\tau_{\epsilon}}$ for all $\delta>0$,
  which implies $\tau_\epsilon \geq -\beta \ln(2\epsilon)$ as desired.

We prove (\ref{Eqn:DeltaLim}) using the spectral
representation of $T^n$.  Since $T$ is geometrically ergodic with
spectral radius $\rho$, for any probability measure $\mu \in L_2(\pi)$
we have
$$\mu T^n = \int_{-\rho}^\rho \lambda^n \mu \xi(d\lambda)  \qquad \forall n \in \N$$ where
$\xi$ is the spectral measure corresponding to $T$ acting on
$L_2(\pi)$ \citep{conw:85}.  Additionally, for any $\delta >0$, either $\xi((\rho
e^{-\delta}, \rho])$ or $\xi([-\rho, -\rho e^{-\delta}))$ is not the
zero operator.  Since $T$ is reversible, one can then apply the
approach in the proof of Theorem 2.1 in \cite{robe:rose:1997} to
construct a signed measure $\omega \in L_2(\pi)$ and a constant $M >
0$ such that $\omega(\X)=0$ and
\begin{align}\label{Eqn:Omega}
\|\omega T^n\|_{TV} \geq M (\rho e^{- \delta})^n \qquad \forall n \in \N.
\end{align}
Without loss of generality $\omega$ can be chosen to have
$\|\omega\|_{TV} \leq 1$.  

Using the triangle inequality, since $\|\omega\|_{TV} \leq 1$ and
  $\omega(\X) = 0$,
\begin{align}\label{Eqn:Delta}
 \|\omega T^n\|_{TV} =  \|\omega T^n - \omega(\X) \pi\|_{TV} \leq
 2\sup_{\mu \in L_2(\pi): \mu(\X) = 1, \mu \geq 0} \| \mu T^n -\pi \|_{TV} \qquad \forall n \in \N
\end{align}
where the supremum is taken over probability measures $\mu\in L_2(\pi)$. Finally, we will show that 
\begin{align}\label{Eqn:TVsup}
 \sup_{\mu \in L_2(\pi): \mu(\X) = 1, \mu \geq 0} \| \mu T^n -\pi \|_{TV} \leq \Delta_n  \qquad \forall n \in \N.
\end{align}
Combining (\ref{Eqn:Omega})-(\ref{Eqn:TVsup}), we have that
$\Delta_n \geq \frac{M}{2}  (\rho e^{- \delta})^n$, so $\Delta_n e^{n (\frac{1}{\beta} +
  \delta)} \geq \frac{M}{2} > 0$ as desired.

To show (\ref{Eqn:TVsup}), Proposition 3(b) of \cite{robe:rose:2004}
implies that for any probability measure $\mu$ on $\X$ and any $n$,
\begin{align*}
\|\mu T^n - \pi\|_{TV} &= \sup_{f: |f|_\infty \leq 1} 
\left| \int f(y) \int \mu(dx) T^n(x, dy) -
\int f(y) \pi(dy) \right|\\
&\leq \sup_{f: |f|_\infty \leq 1} \sup_{x \in \X}  \left| \int f(y) T^n(x, dy) -
\int f(y) \pi(dy) \right|\\
&= \sup_{x \in \X} \sup_{f: |f|_\infty \leq 1} \left| \int f(y) T^n(x, dy) -
\int f(y) \pi(dy) \right|\\
&= \sup_{x \in \X} \|T^n(x,\cdot) - \pi\|_{TV}.
\end{align*}
\end{proof}

\bibliographystyle{ECA_jasa}
\bibliography{/home/fs01/dbw59/docs/samplers/other_peoples_work/bibliography/multimode}

\end{document}